\documentclass[a4paper, 11pt]{article}
\title{ Solitons of Curve Shortening Flow and Vortex Filament Equation\\ \Large UROP+ Final Paper, Summer 2017}
\author{Bernardo Antonio Hernandez Adame\\Mentor: Jiewon Park\\Supervisor: Professor Tobias Colding}
\usepackage{amsmath, amssymb, titling, blindtext, graphicx,epstopdf}
\usepackage[margin = 1in]{geometry}
\graphicspath{{C:/Users/Bernardo H/Documents/MATLAB/EPS}}
\newtheorem{theorem}{Theorem}[subsection]
\newtheorem{lemma}[theorem]{Lemma}
\newtheorem{proposition}[theorem]{Proposition}
\newtheorem{corollary}[theorem]{Corollary}
\newtheorem{definition}[theorem]{Definition}
\newtheorem{claim}[theorem]{Claim}

\newenvironment{proof}[1][Proof]{\begin{trivlist}
		\item[\hskip \labelsep {\bfseries #1}]}{\end{trivlist}}

\newenvironment{example}[1][Example]{\begin{trivlist}
		\item[\hskip \labelsep {\bfseries #1}]}{\end{trivlist}}
\newenvironment{remark}[1][Remark]{\begin{trivlist}
		\item[\hskip \labelsep {\bfseries #1}]}{\end{trivlist}}

\newcommand{\qed}{\nobreak \ifvmode \relax \else
	\ifdim\lastskip<1.5em \hskip-\lastskip
	\hskip1.5em plus0em minus0.5em \fi \nobreak
	\vrule height0.75em width0.5em depth0.25em\fi}

\DeclareMathOperator{\sech}{sech}

\begin{document}
	\begin{titlingpage}
		\maketitle
		\begin{abstract}
			In this paper we explore the nature of self-similar solutions of the Curve Shortening Flow and the Vortex Filament Equation, also known as the Binormal Flow. We explore some of their fundamental conservation properties and describe the behavior of their self-similar solutions. For Curve Shortening Flow we mainly expose the results of Huisken, Grayson, and Halldorsson concerning the equation's basic properties and self-similar solutions in the plane. For the Vortex Filament Equation we present the results by Banica and Vega, Arms and Hama, and Hasimoto. We also derive the evolution equations of the normal, binormal and tangent vectors in the Frenet frame for the vortex filament as well as those of curvature and torsion. We give a proof that circles are the only planar translating self-similar solutions and also derive a system of ordinary differential equations that govern the behavior for rotating self-similar solutions.
		\end{abstract}
	\end{titlingpage}
	\tableofcontents
	
	\newpage
	\section{Preliminaries}
	\subsection{Sobolev Spaces}
	
	We begin with an understanding of Sobolev spaces which are the natural extension of smooth functions that the solutions of PDEs will live in. These spaces serve will also become important in the formulation of dilating self-similar solutions of the Vortex Filament Equation.
	
	 To explain Sobolev Spaces it is first necessary for us to introduce the notion of the weak derivative of a function.
	
	Let $ (\Omega, \Sigma, \mu) $ be a measurable space, $f$ be a function on $\Omega$. If $ f\in L^{p}(\Omega)$ $ (1\leq p<\infty) $ then \[  \lVert f\rVert_{p = }(\int_{\Omega}\lvert f\rvert^{p}d\mu)^{\frac{1}{p}}<\infty\]. If $ p =\infty $ then the norm is instead the essential supremum, defined as follows:
	\begin{definition}
		The essential supremum of a measurable function $ f:\Omega\rightarrow\mathbb{R} $ with measure $ \mu $ is the smallest number $ \alpha $ such that the set \[ \{x\big||f(x)|>\alpha\} \] has measure zero. If no such $ \alpha $ exists then it is taken to be infinity.
	\end{definition}
	 We shall denote this norm as $\lVert f \rVert_{\infty} =\text{ess sup}_{\Omega}\lvert f \rvert $.
	
	To say $ f $ is in $ C^{\infty}_{c}(\Omega) $ means that f is indefinitely differentiable and compactly supported.
	
	\begin{definition}[Weak Derivative]
		Let $ \Omega\subset \mathbb{R}^{n} $ be an open set.  a function $ f\in L^{1}_{loc}(\Omega) $ is said to be weakly differentiable to the $i$th variable if there exists $ g_{i}\in L^{1}_{loc}(\Omega) $ such that \[ \int_{\Omega}f\partial_{i}\phi dx = -\int_{\Omega}g_{i}\phi dx \] for all $ \phi\in C^{\infty}_{c}(\Omega) $. We then call $ g_{i} $ the weak $ i $th partial derivative of $ f $, denoted as usual $ \partial_{i}f $.
	\end{definition}
	
	 It is clear from the definition that the weak derivative coincides with the common pointwise derivative of a continuously differentiable function. However, a function that is not pointwise differentiable almost everywhere can still have a weak derivative. For higher order derivatives the following definition holds.
	
	\begin{definition}
		Let $ \alpha = (\alpha_{1},...,\alpha_{n})\in \mathbb{N}^{n} $ be a multi-index and $ \lvert \alpha\rvert = \alpha_{1}+...+\alpha_{n}=k $. A function $ f\in L^{1}_{loc}(\Omega) $ has a weak derivative of order $ k $, denoted by $ D^{\alpha}f $, if \[ \int_{\Omega} (D^{\alpha}f)\phi dx = (-1)^{\lvert\alpha\rvert}\int_{\Omega}f(D^{\alpha}\phi) dx  \] for all $ \phi\in C^{\infty}_{c}(\Omega) $.
	\end{definition}
	
	We now present some definitions and basic properties of Sobolev spaces without proof. We direct the reader who seeks proofs of theorems as well as a more thorough understanding of Sobolev spaces to chapter 5 of Evans' work on PDEs ([1]) from where the notation and theorems are taken.
	
	\begin{definition}[Sobolev Space]
		We denote with \[ W^{k,p}(\Omega) \] as the space of all locally summable functions $ f:\Omega\rightarrow\mathbf{R} $ such that for each multi-index $ \alpha $ with $ \lvert\alpha\rvert\leq k $, $ D^{\alpha}f $ exists in the weak sense and belongs to $ L^{p}(\Omega) $. If p = 2 it is common to use the notation\[ H^{k}(\Omega) = W^{k,2}(\Omega) \text{ } (k =0,1,...) \]
	\end{definition}

	\begin{definition}[Norm of Sobolev Spaces]
			Let $ f\in W^{k,p}(\Omega) $, we define the norm to be \[ \lVert f \rVert_{W^{k,p}}:=\begin{cases}
				(\Sigma_{\lvert\alpha\rvert\leq k}\int_{\Omega}\lvert D^{\alpha}f\rvert^{p}dx)^{1/p} &(1\leq p <\infty)\\
				\Sigma_{\lvert\alpha\rvert\leq k} \text{ess sup}_{\Omega}\lvert D^{\alpha}f\rvert &(p = \infty) 
				\end{cases}
				\]
	\end{definition}
	
	\begin{definition}[Convergence in Sobolev Spaces]
		(i) Let $ f,\{f_{m}\}^{\infty}_{m = 1} \in W^{k,p}(\Omega) $. We say that $ f_{m} $ converges to $ f $ in $ W^{k,p}(\Omega) $, written \[f_{m}\rightarrow f \mbox{ in } W^{k,p}(\Omega),\] given that \[\lim_{m\rightarrow\infty} \rVert f_{m}-f\lVert_{W^{k,p}(\Omega)} = 0.\]
		
		(ii) We denote by \[ W^{k,p}_{0}(\Omega) \] as the closure of $ C^{\infty}_{c}(\Omega) $ in $ W^{k,p}(\Omega) $.  That is to say that $ f\in W^{k,p}_{0}(\Omega) $ if there exists a sequence $ f_{m}\in C^{\infty}_{c}(\Omega) $ such that $ f_{m}\rightarrow f $ in $ W^{k,p}(\Omega) $.
	\end{definition}
	
	The Fourier transform of a function is defined as follows:
	\begin{definition}
		Let $ f:\mathbb{R}^{n}\rightarrow\mathbb{C} $ be a measurable function. Then the corresponding Fourier transform is defined as \[ \hat{f}(\zeta) = \int_{-\infty}^{\infty}f(x)e^{-2\pi i x\cdot\zeta}d^{n}x \].
	\end{definition}
	
	\begin{theorem}[$ H^{k} $ by Fourier Transform]
		Let $ f\in L^{2}(\mathbf{R}^{n}) $, then $ f\in H^{k}(\mathbf{R}^{n}) $ if and only if \[(1+\lvert y \rvert^{k})\hat{f}\in L^{2}(\mathbf{R}^{n}) \].
		
		In addition there exists a positive constant C such that \[\frac{1}{C}\lVert f\rVert_{H^{k}(\mathbf{R}^{n})}\leq \lVert (1+\lvert y \rvert^{k})\hat{f}\rVert_{L^{2}(\mathbb{R}^{n})}\leq C\lVert f \rVert_{H^{k}(\mathbf{R}^{n})} \] for each $ f\in H^{k}(\mathbf{R}^{n}) $.
	\end{theorem}
	
	\begin{definition}[Non-integer Sobolev Spaces]
		Assume $0 < s <\infty$ and $f\in L^{2}(\mathbf{R}^{n}).$ Then $ f\in H^{s}(\mathbf{R}^{n}) $ if $ (1+\lvert y \rvert^{s})\hat{f}\in L^{2}(\mathbf{R}^{n}) $. For non-integer $ s $ the norm becomes \[\lVert f \rVert_{H^{s}(\mathbf{R}^{n})} := \lVert(1+\lvert y \rvert^{s})\hat{f}\rVert_{L^{2}(\mathbf{R}^{n})}. \]
	\end{definition}
	
	To finish this section we present a small discussion of negative order Sobolev spaces which will be relevant when discussing dilating solutions of the Vortex Filament Equation.
	
	\begin{definition}
		Denote as $ H^{-1}(\Omega) $ as the dual to $ H^{1}(\Omega) $. The norm of this space will be defined as follows \[ \lVert f \rVert_{H^{-1}(\Omega)} = sup\{\langle f, u \rangle|u\in H^{1}_{0}(\Omega), \lVert u\rVert_{H^{1}_{0}(\Omega)}\leq 1 \}. \]
	\end{definition}
	
	\begin{theorem}[Characterization of the Dual Space]
		Assume $ f \in H^{-1}(U) $. Then there exist functions $ f^{0},f^{1},...,f^{1} $ in $ L^{2}(\Omega) $ such that \[ \langle f,\nu\rangle = \int_{\Omega}f^{0}\nu+\sum_{i =1}^{n}f^{i}\nu_{x_{i}}dx \mbox{ \text{(}$ \nu\in H^{1}_{0}(\Omega) $\text{)}} \]
		
		Furthermore, \[ \lVert f \rvert_{H^{-1}(\Omega)} = \mbox{inf}\{ (\int_{\Omega}\sum_{i = 0}^{n}|f^{i}|^{2}dx)^{1/2}| \mbox{ $ f $ satisfies (1) for $ f^{0},...,f^{n}\in L^{2}(\Omega) $} \} \]
	\end{theorem}
	
	For spaces $ H^{-k} $ it coincides with the sense of the previous definition in that they will be the dual of the Sobolev space $ H^{k}_{0} $. For the rest of this paper the reader may assume  that all derivatives are to be taken in the point-wise sense unless otherwise required in the ambient space.
	\subsection{Self Similar Solutions and Geometric PDEs}
	
	In this section we shall discuss the nature of self-similar solutions as well as provide some background on Geometric PDEs, which are the main focus of this paper.
	
	The term self-similar is often used to describe solutions to partial differential equations that demonstrate a particular invariance towards scaling, or in a sense `look the same' at all times. Their importance comes from the fact that they can be used to observe the behavior of a PDE at a singularity by `blowing it up'. In a sense the scaling invariance allows one to create a blow up sequence and rescale time in order to get a clearer picture of the behavior of that PDE at the singularity.
	
	This technique will become clearer in our discussion of Grayson's Theorem when it is used to see the behavior of a point that is `blown up' backwards in time. Here we shall give a basic example of a self-similar solution to the 2-dimensional heat equation, however, for Geometric PDEs, self-similar solutions can be taken to not only be invariant in rescaling but also in translations and rotations. These invariances then give dilating, translating, and rotating self-similar solutions. It is important to note that one can have self-similar solutions that express more than one of these properties.
	
	\begin{theorem}[Self-similar solution of the heat equation]
		Consider the heat equation with the following initial value problem		 \[u_{t}-ku_{xx}= 0\] \[ t>0, -\infty<x<\infty \]
	Then \[ u(x,t) = \frac{1}{\sqrt{4kt}}e^{-x^{2}/(4kt)}\] is a self-similar solution of this equation, i.e. it satisfies $ u_{\alpha}(x,t) = \alpha u(\alpha x, \alpha^{2}t) $ for any $ \alpha>0 $.
	\end{theorem}
	\begin{proof}
		Before beginning we must set the following condition on the solution:
		\begin{align}
			&I(t):=\int_{-\infty}^{\infty}u(x,t)dx<\infty,\\
			&\lim_{x\rightarrow\pm\infty}u_{x}(x,t) = 0
		\end{align}
		
		To find a self-similar solution we first begin by finding a solution $ u_{\alpha}(x,t) $ of the form $ u_{\alpha}(x,t) = \alpha u(\alpha x, \alpha^{2}t) $, such that $ u(x,t) = u_{\alpha}(x,t) $ for all $ \alpha, t>0, x $. The previous condition satisfies the invariance of a self-similar solution, and this method of rescaling variables is called parabolic rescaling.
		
		If we then take $ \alpha = t^{-1/2} $ to get rid of the time variable, we have that \[ u(x,t) = \frac{1}{\sqrt{t}}u(\frac{x}{\sqrt{t}},1) = \frac{1}{\sqrt{t}}\phi(\frac{x}{\sqrt{t}}) \].
		
		Then we get the following equations by taking the derivatives:\begin{align} u_{t} &= -\frac{1}{2}t^{-3/2}[\phi(\frac{x}{\sqrt{t}})+\frac{x}{\sqrt{t}}\phi'(\frac{x}{\sqrt{t}})]\\
		u_{x}&=\frac{1}{t}\phi'(\frac{x}{\sqrt{t}})\\
		u_{xx} &=t^{-3/2}\phi''(\frac{x}{\sqrt{t}})
		 \end{align}
		 
		 We then make the substitution $ \zeta = \frac{x}{\sqrt{t}} $ and plug into our heat equation to obtain the following ODE:\[ -\frac{1}{2}[\phi(\zeta)+\zeta\phi'(\zeta)]-k\phi''(\zeta) = 0 \].
		 
		 Finally, solving this ODE gives \[ \phi(\zeta) = Ce^{-\frac{\zeta^{2}}{4k}} \]
		 so that $ u(x,t) = \frac{1}{\sqrt{4\pi kt}}e^{-\frac{\zeta^{2}}{4k}} $ where $ C = \frac{1}{\sqrt{4k\pi}} $ in order to normalize $ I(t) $.\qed
		 
	\end{proof}
			 The method used in the previous proof of parabolically rescaling the variable is used when attempting to find dilating solutions of the vortex filament equation.
			 
			 In the description of geometric properties of manifolds there are often situations that arise which are modeled with a system of PDEs, this then allows us to use the tools of PDE theory to be able to investigate the geometric, analytic, and topological properties of the objects these equations describe. Geometric PDEs have a wide range of applications, from aiding in the solution of previously open problems such as the Poincar\'e conjecture and the differentiable sphere theorem, to applications in image and sound processing.
			 
			 In this paper we mainly concern ourselves with Geometric Flows. These are systems of equations that describe the deformation of metrics on Riemannian manifolds driven by the geometric quantities such as curvature, volumes, etc. In this paper we present a discussion on two quasi-parabolic geometric flows: Curve Shortening Flow and the Vortex Filament Equation.
	\section{Curve Shortening Flow}
		
		Curve Shortening Flow (CSF) is a geometric quasi-parabolic equation for curves	that serves as an analogous flow to the Vortex Filament Equation. This geometric flow equation has been extensively studied and we have derived many of its more important properties such as monotonicity formulas, and maximum principle estimates that combine with the blowup analysis we mentioned earlier while discussing self-similar solutions to give Huisken's ([3]) proof of Grayson's Theorem.

	\begin{definition}[Curve Shortening Flow]
		A family of embedded curves $ \{\Gamma_t\subset \mathbf{R}^2 \}_{t \in I }$ moves by curve shortening flow if the normal velocity at each point is given by the curvature vector: 
			\begin{equation}
				\partial_t p = \vec{\kappa}(p)
			\end{equation}
		for all $ p \in \Gamma_t $ and all $t \in I $. Here, $I$ is an interval, $\partial_t p $ is the normal velocity at $p$, and $\vec{\kappa}(p)$ is the curvature vector at $p$.
	\end{definition}
	
	We now provide an important example of an explicit solution which also happens to be the only self-similar translating solution of curves evolving under CSF:
	\begin{example}[Grim Reaper Curve]
		Take $ y(x,t) =t-\log \cos(x) $ then under CSF this curve moves upwards without changing its shape. Any curve similar, by either scaling translation or rotation, to the grim reaper is also translated in the direction of the axis of symmetry without changing shape or orientation, satisfying what is to be expected of a self-similar solution. Interestingly it is the only curve with this property ([13]).
	\end{example}
	\subsection{Properties of Curve Shortening Flow}
	We will focus on the evolution of closed embedded curves. For this section we mainly present the theorems of Haslhofer's notes on CSF ([2]) and his discussion on blow-up methods for singularities. The main part of this section will be the presentation by Haslhofer of Huisken's proof of Grayson's Theorem ([6]) from his lecture notes ([2]), for which we shall present all the necessary machinery without proof. Although there are various other proofs of Grayson's Theorem, some more geometric than others and requiring less machinery, the purpose of this one is to demonstrate the utility of self-similar solutions and why we seek to discover them.
	
	CSF can also be rewritten by taking 
	\begin{equation}
	\gamma = \gamma(\cdot , t) : S^1 \times I \rightarrow \mathbb{R}^2
	\end{equation}
	with $\Gamma_t = \gamma (S^1, t)$. Setting $p = \gamma (x,t)$, the equation transforms into 
	\begin{equation}
	\partial_t\gamma(x,t) = \kappa(x,t)N(x,t).
	\end{equation}
	
	\begin{remark}
		The evolution can also be written in the form 
		\begin{equation}
		\partial_t\gamma = \partial_s^2\gamma,
		\end{equation}
		where s denotes arc length if one were to transform from the Frenet frame of reference to the Cartesian plane.
	\end{remark}
	
	First we present some facts of CSF and derive some of its basic properties such as evolution for the length functional.
	
	\begin{theorem}[Evolution of Arclength]
		Define $ L(t) := \int ds $ to be the arclength functional. A curve $ \gamma $ evolving under CSF has decreasing arclength which obeys the following evolution equation:\begin{equation}
			\partial_{t}L = -\int\kappa^{2}ds.
		\end{equation}
	\end{theorem}
	
	\begin{proof}
		From calculus we have that $ L(t) = \int ds = \int_{S^{1}}\langle\gamma_{x},\gamma_{x} \rangle^{1/2}dx $ where $ x\in S^{1} $.
		
		We take the derivative with respect to time and proceed to differentiate under the integral to obtain\begin{align} \partial_{t}L &= \int_{S^{1}}\langle\partial_{xt}\gamma,T\rangle du\\
		&=\int_{S^{1}}\langle\partial_{tx}\gamma,T\rangle du\\
		&=\int_{S^{1}}\langle \partial_{x}(\kappa N),T\rangle dx
		\end{align}
		were $ T $  and $ N $ are the tangent and normal vectors, respectively. From the Frenet equations we have that $ \partial_{x}T = s'(x)\kappa N $ and $ \partial_{x}N=-s'(x) \kappa T $. So substituting this gives that the only component not zero by orthogonality is $ -\kappa^{2} $, turning (13) into \[ \partial_{t}L =-\int_{S^{1}}\kappa^{2}s'(x)dx = -\int\kappa^{2}ds \] as desired.
	\end{proof}
	
	\begin{remark}
		This could also be easily derived from the first variation formula of arc-length, which says that if a curve moves with normal velocity $ v $ the length of the curve changes by $ -\int \kappa v $. This theorem also shows why it is considered that curves under CSF decrease their length most efficiently.
	\end{remark}
	\begin{proposition}[Evolution of Curvature]
		Suppose that $ \gamma $ is a curve that satisfies CSF, then its curvature satisfies the evolution equation \begin{equation}
			\kappa_{t} = \kappa_{ss}+\kappa^{3}.
		\end{equation}
	\end{proposition}
	\begin{proof}
		We work here with a parametrization such that $ |\partial_{x}\gamma| =1 $ and $ \langle\partial^{2}_{x}\gamma,N\rangle =0$ at the point $(x,t) $. By definition $ \kappa = |\partial_{x}\gamma|^{-2}\langle\partial^{2}_{x}\gamma,N\rangle $ compute \begin{equation}
			\kappa_{t} = \partial_{t}\langle\partial_{xx}\gamma,N\rangle-2\langle T,\partial_{xt}\gamma,N\rangle\\
			=\langle\partial_{tx}\gamma,N\rangle-2\kappa\langle T,\partial_{tx}\gamma\rangle
		\end{equation}
		since $ N_{t} $ moves in the tangent direction and $ \partial_{xx}\gamma =\kappa N$. Using the orthogonality relations between tangent and normal vectors and plugging in (8) we obtain\begin{align}
			\partial_{t}\kappa &= \partial_{xx}\kappa+\kappa\langle\partial_{xx}N,N\rangle-2\kappa^{2}\langle T,\partial_{x}N\rangle\\
			&= \partial_{xx}\kappa-\kappa\langle\partial_{x}N,\partial_{x}N\rangle +2\kappa^{3}\\
			&= \kappa_{ss}+\kappa^{3}
		\end{align}
		where we used that $ \partial_{x}N = -\kappa T $.
	\end{proof}
	
	We now present the following corollary without proof:
	\begin{corollary}[Conservation of Convexity]
		Convexity is preserved under curve shortening flow, i.e. if $ \kappa>0 $ at $ t = 0 $ then $ \kappa>0 $ for all $ t\in[0,\infty) $.
	\end{corollary}
	
	We now present an important formula without proof derived by Huisken ([4]) that is useful when trying to look at blow-up analysis on singularities of CSF as described in Section 1.2 since this equation is invariant under parabolic scaling.
	
	\begin{definition}[Heat Kernel]
		Let $ X_{0} = (x_{0},t_{0}) $ and denote by $ \rho_{X_{0}}(x,t) $ the backwards Heat Kernel. We then define it as\begin{equation}
			\rho_{X_{0}}(x,t) = (4\pi(t_{0}-t))^{-1/2}e^{-\frac{|x-x_{0}|^{2}}{4(t_{0}-t)}}
		\end{equation}
		for $ t<t_{0} $.
	\end{definition}
	
	\begin{theorem}[Huisken's Monotonicity Formula]
		Let $ \{\Gamma_{t}\} $ be a family of curves that move by CSF, then 
		\begin{equation}
			\frac{d}{dt}\int_{\Gamma_{t}} \rho_{X_{0}}ds =-\int_{\Gamma_{t}}\Big|\kappa+\frac{\langle\gamma,N\rangle}{2(t_{0}-t)} \Big|^{2}\rho_{X_{0}}ds\qquad(t<t_{0`}) 
		\end{equation}
	\end{theorem}
	
	Although the goal of the following theorem is to prove existence and uniqueness for solutions of CSF its importance is in the maximal existence time curvature.
	
	\begin{theorem}[Existence and Uniqueness]
		Let $ \gamma_{0}:S^{1}\rightarrow\mathbb{R}^{2} $ be an embedded curve. Then there exists a unique smooth solution $ \gamma:S^{1}\times[0,T)\rightarrow\mathbb{R}^{2} $ of curve shortening flow defined on a maximal interval $ [0,T) $. The maximal existence time is characterized by \begin{equation}
			\sup_{S^{1}\times[0,T)}|\kappa| = \infty
		\end{equation}
	\end{theorem}
	
	\begin{theorem}[Local Regularity Theorem $ \text{([7],[8])} $]
		Let $ X = (x,t) $ be a point in space-time and $ P_{r}(X)=B_{r}(x)\times(t-r^{2},t] $ for the parabolic ball with center $ X $ and radius $ r $. There exist universal constants $ \epsilon >0 $ and $ C<\infty $ with the following property. If $\{\Gamma_{t}\subset\mathbb{R}^{2}\}_{t\in(t_{0}-2r^{2},t]} $ is a curve shortening flow with \begin{equation}
			\sup_{\bar{X_{0}}\in P_{r}(X_{0})}\Theta\big(\{\Gamma _{t}\}, \bar{X_{0}},r \big) :=\int_{\Gamma_{t_{0}}-r^{2}}\rho_{(x_{0},t_{0})}ds <1+\epsilon,
		\end{equation}
		then \begin{equation}
			\sup_{P_{r/2}(X_{0})}|\kappa|\leq\frac{C}{r}
		\end{equation}
	\end{theorem}
	
	The importance of Theorem 2.1.7 shall be when we create blowup sequences that are close to a circle  as we approach the singularity, and won't allow for it to be a straight line. This is a fact we shall use to show that the curve must become a circle in the proof of Grayson's Theorem.
	
	\begin{theorem}[Hamilton's Harnack Inequality$ \text{([9])} $]
		If $ \{ \Gamma_{t}\subset\mathbf{R}^{2} \}_{t\in[0,T)} $ is a convex solution of CSF then $\frac{\kappa_{t}}{\kappa}-\frac{\kappa_{s}^{2}}{\kappa^{2}}+\frac{1}{2t}\geq0 $
	\end{theorem}
	
	We now present Huisken's distance comparison principle between the extrinsic and intrinsic distances ([5]). It shows that embededness is preserved by CSF.
	
	\begin{theorem}[Huisken's Distance Comparison Principle $ \text{([5])} $]
		If a family of closed embedded curves $ X $ in the plane evolves by CSF, then the following equation \begin{equation}
			R(t) :=\sup_{x\neq y} \frac{L(t)}{\pi d(x,y,t)}sin\frac{\pi l(x,y,t)}{L(t)},
		\end{equation}
		where $ L(t) $ is the total length of the curve, $ l(x,y,t) $ is the intrinsic distance between $ X(x,t) $ and $ X(y,t) $, and $ d(x,y,t) = |X(x,t)-X(y,t)| $, is non-increasing in time.
	\end{theorem}
	
	The important part to take away from this theorem is that the intrinsic and extrinsic distance equation is bounded by $ R(0)<\infty $. Particularly this implies that the grim reaper solutions cannot arise as a blow-ip limit of CSF closed embedded curves.
	
	We finish this section with Huisken's proof of Grayson's Theorem ([6]) via the method of singularity blow-up which we present in full to demonstrate the utility of self-similar solutions. First we make the following definition of a blow-up point for CSF.
	\begin{definition}[Blow-up Point]
		We say that $ x_{0}\in\mathbb{R}^{2} $ is a blowup point if there are sequences $ t_{i}\rightarrow T $, $ p_{i}\in \Gamma_{t} $ such that $ |\kappa|(p_{i})\rightarrow\infty $ and $ p_{i}\rightarrow x_{0}$, i.e it achieves the curvature of the maximal existence point as defined in Theorem 2.1.6.
	\end{definition}
	
	\begin{theorem}[Grayson's Theorem $ \text{([6])} $]\label{Grayson}
		If $ \Gamma\subset\mathbb{R}^{2} $ is a closed embedded curve, then the curve shortening flow $ \{\Gamma_{t}\}_{t\in[0,T)} $ with $ \Gamma_{0}=\Gamma $ exists until $ T = \frac{A_{\Gamma}}{2\pi} $ and converges for $ t\rightarrow T $ to a round point, i.e. there exists a unique point $ x_{0}\in\mathbb{R}^{2} $ such that the rescaled flows \begin{equation}
			\Gamma^{\lambda}_{t}:=\lambda\cdot\big( \Gamma_{T+\lambda^{-2}t}-x_{0}\big)
		\end{equation} 
	converge for $ \lambda\rightarrow\infty $ to the round shrinking circle $ \{\partial B_{\sqrt{-2t}} \}_{t\in(-\infty,0)} $
	\end{theorem}
	
	To prove this theorem the following lemma is needed\begin{lemma}
		Along CSF we have that \begin{equation}
			\frac{d}{dt}\int_{\Gamma_{t}}|\kappa|ds = -2\sum_{x:\kappa(x,t) = 0}|\kappa_{s}|(x,t)
		\end{equation}
	\end{lemma}
	
	\begin{proof}
		Since solutions of CSF are analytic there are only a finite number of inflection points, giving \[ \frac{d}{dt}\Big( \int_{\{\kappa\geq0\}}\kappa ds-\int_{\{\kappa\leq0\}}\kappa ds \Big) = \int_{\{\kappa\geq0\}}\kappa_{ss} ds-\int_{\{\kappa\leq0\}}\kappa_{ss} ds \] and integrating by parts gives the results.
	\end{proof}
	\begin{proof}
		(Proof of Theorem \ref{Grayson}) Let $ T<\infty $ be the maximal existence time of CSF starting at $ \Gamma $. Suppose towards a contradiction \begin{equation}
			\limsup_{t\rightarrow T}\Big( (T-t)\max_{\Gamma_{t}}\kappa^{2} \Big) = \infty
		\end{equation}
		i.e. we have a type II blow-up. For any integer $ k\geq1/T $ we let $ t_{k}\in [0,T-\frac{1}{k}], x_{k}\in S^{1}$ be such that \begin{equation}
			\kappa^{2}(x_{k},t_{k})(T-1/k-t_{k}) = \max_{t\leq T-1/k,x\in S^{1}}\kappa^{2}(x,t)(T-1/k-t).
		\end{equation}
		
		We also set \begin{equation*}
			\lambda_{k}=\kappa(x_{k},t_{k}),\quad t^{(0)}_{k} = -\lambda^{2}_{k}t_{k}, \quad t^{(1)}_{k} = \lambda^{2}_{k}(T-1/k-t_{k}).
		\end{equation*}
		
		We can then say, thanks to (27), that for any $ M\leq \infty $ there exist $ \bar{t}< T $ and $ \bar{x}\in S^{1} $ such that $ \kappa^{2}(\bar{x},\bar{t})(T-\bar{t})
		>2M $. For $ k $ large enough we have \begin{equation}
			\bar{t}<T-1/k,\quad\kappa^{2}(\bar{x},\bar{t})(T-\bar{t}-1/k)>M.
		\end{equation}
		
		Then it follows that \begin{equation}
			t^{(1)}_{k} = \kappa^{2}(x_{k},t_{k})(T-1/k-t_{k})\geq\kappa^{2}(\bar{x},\bar{t})(T-1/k-\bar{t})>M.
		\end{equation}
		So then since $ t_{k}^{(1)} $ is increasing and $ M $ is arbitrary, this implies $ t^{(1)}_{k}\rightarrow\infty $, so then $ \lambda_{k}\rightarrow\infty,t_{k}\rightarrow T $ and $ t^{(0)}_{k}\rightarrow-\infty $.
		
		Then consider the sequence  of the rescaled flow \begin{equation}
			\Gamma_{t}^{k} = \lambda_{k}\cdot\Big( \Gamma_{t_{k}+\lambda_{k}^{-2}t}-x_{k} \Big),\qquad t\in[t_{k}^{(0)}, t_{k}^{(1)}).
		\end{equation}
		and we find that by construction, $ \Gamma^{k}_{t} $ has $ \kappa_{k}=1 $ at $ t =0 $ at the origin. Then by our choice $ (x_{k},t_{k}) $ implies \begin{equation}
			\kappa^{2}_{k}(x,t)\leq\frac{T-1/k-t_{k}}{T-1/k-t_{k}-\lambda^{2}_{k}t} = \frac{t^{(1)}_{k}}{t^{(1)}_{k}-t},\quad t\in[t_{k}^{(0)},t^{(1)}_{k}).
		\end{equation}
		Then we have that after passing to a subsequence, we get the smooth limit $ \{\Gamma_{t}^{\infty} \}t\in(-\infty,\infty) $. Then we have by our previous formulation that the limit $ \kappa = 1 $ at the time 0 at the origin, and $ \kappa^{2}\leq1 $ at every point for all time. Then by the previous Lemma the limit satisfies \begin{equation}
			\int_{-\infty}^{\infty}\sum_{x:\kappa(x,t) = 0}|\kappa_{s}|(x,t)dt = 0
		\end{equation}
		i.e. if $ \kappa = 0 $ then $ \kappa_{s} = 0 $ as well. Then, using the evolution equations and analyticity this implies that $ \{\Gamma_{t}^{\infty}\}_{t\in(-\infty,\infty)} $ is a straight line, a contradiction.
		
		This then gives that $ \kappa>0 $, and by equality in the case of Hamilton's Harnack Inequality, and the fact that a translating soliton for CSF must be a grim reaper curve, this contradicts the bound for the ratio between intrinsic and extrinsic distance.
		
		So then we have shown the a type I blow-up rate \begin{equation}
			\limsup_{t\rightarrow T}\big( (T-t)\max_{\Gamma_{t}}\kappa^{2} \big)<\infty
		\end{equation}
		
		To finish the discussion we simply need to prove the following claim:\begin{claim}
			Define the parabolic scaling of a family of curves that satisfy CSF $\{\Gamma^{\lambda}_{t}\}_{t\in[-\lambda^{2}T,0)} $ where $ \Gamma^{\lambda}_{t}:=\lambda\cdot(\Gamma_{T+\lambda^{-2}t}-x_{0}) $. Then, for the limit as $ \lambda\rightarrow\infty $, these converge smoothly to a family of round shrinking circles $ \{ \partial B_{\sqrt{-2t}}(0) \}_{t\in(-\infty,0)}. $
		\end{claim}
		To finish the proof we rescale the blowup rate and try to show that the limit intersects families of round shrinking circles.
		\begin{equation}
			\max_{\Gamma^{\lambda}_{t}}|\kappa|\leq\frac{C}{\sqrt{-t}},\qquad t\in[-\lambda^{2}T,0).
		\end{equation}
		
		We have in our previous argument found a subsequence of $ \lambda_{k} $ such that $ \{ \Gamma^{\lambda_{k_{i}}}_{t} \} $ converges smoothly to a limit. By construction, the limit is an ancient solution of CSF. Then we can use the definition of blow-up points and comparing with round shrinking circles we can see that $ \Gamma^{\lambda}_{-1}\cap B_{2}(x_{0}) \neq\emptyset$ for $ \lambda $ large enough. So then the limit is not empty. Then by Huisken's Monotonicity Formula for all $ t_{1}<t_{2}<0 $ it gives that \begin{equation}
			\int_{t_{1}}^{t_{2}}\int_{\Gamma_{t}^{\lambda}}\Big|\kappa+\frac{\langle\gamma,N\rangle}{2(t_{0}-t)} \Big|^{2}\rho dsdt =-\Big[\int_{\Gamma_{t}^{\lambda}} \rho_{X_{0}}ds\Big]^{T-t_{2}/\lambda^{2}}_{T-t_{1}/\lambda^{2}}\rightarrow 0
		\end{equation}
		as $ \lambda\rightarrow 0 $. So then the limit is self-similarly shrinking and completely determined by its slice at $ t=-1 $ satisfying \begin{equation}
			\kappa+\frac{\langle \gamma,N\rangle}{2} = 0.
		\end{equation}
		Then by the local regularity theorem it cannot be a straight line, so then $ \Gamma_{-1} $ must be a circle of radius $\sqrt{2} $ which completes the theorem.
	\end{proof}
	\subsection{Self-Similar Solutions of Curve Shortening Flow}
	In this section we summarize the work of Halldorsson ([10]) who  gave the classification of all self-similar embedded curves that evolve under CSF in the plane. We also direct the reader to the work of Altschuler et al ([11]) who classified the solitons of CSF in $ \mathbb{R}^{n} $ through the creation of a group acting on curves evolving under CSF that produced self-similar solutions that turned the PDE into an ordinary differential equation.
	
	We return now to the work of Halldorsson and follow his steps to derive the family of ODE's that provide self-similar solutions and state his theorems in full. In [10] self-similar solutions were classified under the following classifications:\begin{itemize}
		\item Translating Curves: Only the Grim Reaper curve
		\item Expanding Curves: A one dimensional family of curves. Each is properly embedded and asymptotic to the boundary of a cone.
		\item Shrinking Curves: A one-dimensional family of curves. Each is contained in an annulus and consists of identical excursions between both boundaries.
		\item Rotating Curves: A one-dimensional family of curves. Each is properly embedded and spirals out to infinity.
		\item Rotating and Expanding: A two dimensional family of curves, properly embedded that spiral out to infinity.
		\item Rotating and Shrinking Curves: A two-dimensional family of curves, with an end asymptotic to a circle and the other either similarly asymptotic or spiraling out to infinity.
	\end{itemize}
	
	Consider instead of CSF acting on curves mapping $\mathbb{R}  $ to  $ \mathbb{R}^{2}$ that they were mapped instead to the $ \mathbb{C} $, the change of space allows one to simplify the action of rotations on the curve. Now, following the steps of Halldorsson let $ X:\mathbb{R}\times I\rightarrow\mathbb{C} $ be a curve evolving under CSF, and being a self-similar solution it is of the form \begin{equation}
		\hat{X}(x,t) = g(t)e^{if(t)}X(x)+H(t)
	\end{equation}
	where $ I $ is an interval 0, and $ f,g,H $ are all differentiable functions such that $ f(0) = 0,g(0) =1 $ and $ H(0) =0 $ so that $ \hat{X}(x,0) = X(x) $.
	
	In $ \mathbb{C} $ we define the normal vector $ N(x,t) =iT(x,t) $. Then plugging in (38) into the definition of CSF gives\begin{equation}
		g^{2}(t)f'(t)\langle X(x),T(x)\rangle +g(t)g'(t)\langle X,N\rangle+g(t)\langle e^{-if(t)}H'(t),N\rangle = \kappa(x).
	\end{equation} 
	
	Since this equation must hold for all $ (x,t)\in\mathbb{R}\times I $.  For $ t =0 $ the curve must then satisfy the following ODE:\begin{equation}
		A\langle X,T\rangle+B\langle X,N\rangle+\langle C,N\rangle=\kappa(x)
	\end{equation}
	where $ A =f'(0),\quad B =g'(0),\quad\text{and}\quad C=H'(0) $.
	
	Looking at solutions where the dilation term vanishes gives \begin{equation}
		A\langle X,T\rangle+B\langle X,N\rangle =\kappa.
	\end{equation}
	We now look to prove the following theorem:
		\begin{theorem}[Existence of Embedded Self-Similar Curves $ \text{[10]} $]
			For each value of A and B there exists an immersed curve $ X $ satisfying (41).
		\end{theorem}
	
	To ensure that (40) exists for all time choose $ g^{2}f' =A $ and $ g(t)g'(t)=B $ for all $ t\in I $. We choose $ f $ and $ g $ to be \[ f(t )=\begin{cases}
		  \frac{A}{2B}\log(2Bt+1)\quad&\text{if}\quad B\neq0,\\
		  At\quad&\text{if}\quad B=0
	\end{cases}\] and \begin{equation}
		g(t) =\sqrt{2Bt+1}.
	\end{equation}
	Thanks to these equations we see that $ X $ rotates around the origin, with exception of if $ A=0 $ of course, and dilates outwards for $ B>0 $ and inwards if $ B<0 $. Including the rotation term $ C $ only causes the curve to screw-dilate around  the point $ \frac{-C}{B+iA} $. For curves that only translate we have that it can only give a Grim Reaper Curve as was shown in [13].
	
	Now consider a parametrization by arclength and use the Frenet equations to obtain the following relations\begin{align}
		\frac{d}{ds}\langle X,T\rangle &= 1+\kappa\langle X,N\rangle,\\
		\frac{d}{ds}\langle X,N\rangle &=-\kappa\langle X,N\rangle.
	\end{align}
	
	Then define the equations\begin{align}
		x &= A\langle X,T\rangle+B\langle X,N\rangle\\
		y &= -N\langle X,T\rangle+A\langle X,N\rangle\\
		x+iy &= (A-iB)(\langle X,T\rangle+i\langle X,N\rangle),		
	\end{align}
	that satisfy \begin{align}
		x' = \kappa y+A,\\
		y' = -\kappa x-B.  
	\end{align}
	and from these equations we can rewrite $ X $ to be \begin{equation}
		X =e^{i\theta(s)}\frac{x+iy}{A-iB}
	\end{equation}
	where $ \theta(s) = \int_{0}^{s}\kappa(z)dz+\theta_{0} $ and $ T(0) = e^{i\theta_{0}} $. Since we now look for curves that satisfy $ x =\kappa $, we substitute this into (48) and (49), as well as the definition of $ \theta $ to obtain \begin{align}
		x'&=xy+A\\
		y'&=-x^{2}-B
	\end{align}
	with initial conditions $ x_{0} $ and $ y_{0} $.
	
	We then see that by these equations \[ X' = e^{i\theta}, \] so that then $ T = e^{i\theta} $ and $ X $ is parametrized by arclength, so then $ \kappa = \theta' =x $. To finish the proof we simply run through the following calculation\begin{align}
		A\langle X,T\rangle+B\langle X,N\rangle &=\langle X,(A+iB)Y\rangle\\
		&=Re(X(A-iB)e^{-i\theta})\\
		&=x\\
		&=\kappa
	\end{align}
	finishing the theorem.
	
	All the possible values of $ A,B $ create then a 2 parameter family of ODEs that govern the creation of self-similar curves for CSF, and are classified as in the beginning of this section. The differing possible values for $ A $ and $ B $ are what splits CSF into the following 4 cases:\begin{itemize}
		\item $ A\neq0 $ and $ B\geq 0$ gives rotating expanding curves.
		\item $ A\neq0 $ and $ B< 0$ gives rotating shrinking curves.
		\item $ A=0 $ and $ B< 0$ gives only shrinking curves.
		\item $ A=0 $ and $ B> 0$ gives only expanding curves.
	\end{itemize}
	
	Halldorson then goes on a case by case basis of solving the ODEs and finding their significant properties and deriving the following theorems which we present without proof and direct the reader to [10] in order to not only see their proof but also the accompanying graphics.
	
	\begin{theorem}[$ A\neq0 $ and $ B\geq 0$]
		The curves are properly embedded, have one point closest to the
		origin and consist of two arms coming out from this point which strictly go away from the origin to infinity. Each arm has infinite total curvature and spirals infinitely many circles around the origin. The curvature goes to 0 along each arm, and
		their limiting growing direction is $ B + iA $ times the location.
		
		The curves form a one-dimensional family parametrized by their distance to the origin, which can take on any value in $[ 0, \infty) $.
		
		If $ B = 0 $, then under the CSF the curves rotate forever with constant angula speed A.
		
		If $ B > 0 $, then under the CSF the curves rotate and expand forever with angular	function $ \frac{A}{2B}\log(2Bt + 1) $ and scaling function $ 
		\sqrt{2Bt + 1}. $
	\end{theorem}
	
	\begin{theorem}[$ A\neq0 $ and $ B< 0$]
		In this case there are two types of curves.
		1) Curves such that the limiting behavior when going along the curve in each	direction is wrapping around the circle of radius $\frac{1}{\sqrt{-B}} $, clockwise if $ A < 0 $ and
		counterclockwise if $ A > 0 $. These curves form a one-dimensional family.
		
		2) Curves such that the curvature never changes sign and the two ends behave very differently. One end wraps around the circle with radius $ \frac{1}{\sqrt{-B}} $
		in its limiting
		behavior, clockwise if $ A < 0 $ and counterclockwise if $ A > 0 $. The other end spirals infinitely many circles around the origin out to infinity and has infinite total curvature. The curvature goes to 0 along it, and its limiting growing direction is
		$ -B - iA  $ times the location. There is at least one curve of this type, and we call it the comet spiral.
	
		These curves rotate and shrink with angular function $\frac{A}{2B}\log(2Bt+1) $ and scaling
		function $ \sqrt{2Bt + 1} $ under the CSF. A singularity forms at time $ t = -\frac{1}{2B} $. The curves of type 1 are bounded, so they disappear into the origin.
	\end{theorem}
	\begin{theorem}[$ A=0 $ and $ B< 0$]
		Each of the curves is contained in an annulus around the origin
		and consists of a series of identical excursions between the two boundaries
		of the annulus. The curvature is an increasing function of the radius and never changes sign. The inner and outer radii of the annulus, $ r_{min} $ and $ r_{max} $, satisfy
		$r_{min}\exp(Br_{min}^{2}/2)=r_{max}\exp(Br_{max}^{2}/2) $ and take on every value in $ (0, \frac{1}{\sqrt{-B}}] $ and $ [\frac{1}{\sqrt{-B}},\infty) $,
		respectively.
		
		The curves form a one-dimensional family parametrized by $ r_{min} $ and are divided
		into two sets:
		
		1) Closed curves, i.e., immersed $ \mathbb{S}^{1} $	(Abresch-Langer curves ([12])). In addition to	the circle, there is a curve with rotation number p which touches each boundary of	the annulus $ q $ times for each pair of mutually prime positive integers $ p, q $ such that $ \frac{1}{2}<\frac{p}{q} $.
		
		2) Curves whose image is dense in the annulus.
		
		Under the CSF these curves shrink with scaling function $g(t) =\sqrt{2Bt + 1} $ until
		they disappear into the origin at time $ t = -\frac{1}{2B}$.
	\end{theorem}
	\begin{theorem}[$ A=0 $ and $ B \geq 0 $]
		Each of the curves is convex, properly embedded and asymptotic to
		the boundary of a cone with vertex at the origin. It is the graph of an even function.
		
		The curves form a one-dimensional family parametrized by their distance to the
		origin, which can take on any value in $ [0, \infty) $.
		
		Under the CSF these curves expand forever as governed by the scaling function $ g(t) = \sqrt{2Bt + 1} $.
	\end{theorem}
	\section{The Vortex Filament Equation}
	The Vortex Filament Equation (VFE) \[ \partial_{t}\gamma = \partial_{s}\gamma\times\partial_{ss}\gamma, \] where $ s $ is the arc-length, arises in the consideration of vortices with infinitesimal thickness of size and whose effects at infinity can be ignored. It has applications in the description on the shape of vortices and the interaction of vortex lines and has applications in aerodynamics as well as high energy quantum fluids. Here we present a derivation of the equation by the Local Induction Principle as given by Arms and Hama ([14]).
	
	Hama and Arms first consider the Biot-Savart law:\begin{equation}
		d\mathbf{q}_{ij} = -\frac{k}{4\pi}\mathbf{r}^{-3}_{ij}\frac{\partial \mathbf{r}_{ij}}{\partial s_{i}}\times \mathbf{r}_{ij}ds_{j}.
	\end{equation}
	Here $ k $ is a scalar and is the strength of the vortex, $ d\mathbf{q}_{ij} $ is the induced velocity at the point $ r_{i} $ by the vortex segment $ ds_{j} $ at the point $ \mathbf{r}_{i} $, and $ \mathbf{r}_{ij} $ is the vector distance between the points $ \mathbf{r}_{i} $ and $ \mathbf{r}_{j} $.
	
	Begin by expanding through a Taylor Series the vector \[  \mathbf{r}_{ij}(\zeta,t) = \mathbf{r}_{i}(s_{i},t) -\mathbf{r}_{j}(s_{i}+\zeta,t). \] Assuming that $ \zeta $ is small then the expression becomes:\begin{equation}
		\mathbf{r}_{ij}(\zeta) = \mathbf{a}_{1}\zeta+\mathbf{a}_{2}\zeta^{2}+...
	\end{equation}
	where the substitutions \begin{equation}
		\mathbf{a}_{1} = \partial_{\zeta}\mathbf{r}_{ij},
		\quad\mathbf{a}_{2}= \partial^{2}_{\zeta}\mathbf{r}_{ij},\text{... at }\zeta = 0
	\end{equation}
	are made.
	
	This gives that $ \partial\mathbf{r}_{ij}/\partial s_{i} = \partial\mathbf{r}_{ij}/\partial\zeta = \mathbf{a}_{1}+2\mathbf{a}_{2}\zeta+... $ and \begin{align}
		-\partial\mathbf{r}_{ij}/\partial s_{i}\times\mathbf{r}_{ij} &= (\mathbf{a}_{1}\zeta+\mathbf{a}_{2}\zeta^{2}+...)\times(\mathbf{a}_{1}+2\mathbf{a}_{2}\zeta+...)\\
	&=(\mathbf{a}_{1}\times\mathbf{a}_{2})\zeta^{2}+O(\zeta^{3})\\
	&= (\mathbf{a}_{1}\times\mathbf{a}_{2})|\zeta|^{2}.
	\end{align}
	
	Similarly to find the value of the distance vector, \begin{align}
		|\mathbf{r}_{ij}|^{2} &= |(\mathbf{a}_{1}\zeta+\mathbf{a}_{2}\zeta^{2}+...)^{2}|\\
		&= |\mathbf{a}_{1}|^{2}|\zeta|^{2}+2\mathbf{a}_{1}\cdot\mathbf{a}_{2}\zeta^{3}+...\\
		r_{ij} &= |\mathbf{a}_{1}||\zeta|\Big(1+2\frac{\mathbf{a}_{1}\cdot\mathbf{a}_{2}}{|\mathbf{a}_{1}|^{2}}\zeta+... \Big)
	\end{align}
	so that we can then take the exponent and expand the binomial to obtain \[ r_{ij}^{-3} = |\mathbf{a}_{1}|^{-3}|\zeta|^{-3}\Big(1-3\frac{\mathbf{a}_{1}\cdot\mathbf{a}_{2}}{|\mathbf{a}_{1}|^{2}}\zeta+... \Big). \]
	
	Finally, this gives the expression \begin{equation}
		\mathbf{q}_{ij} = \frac{k}{4\pi}\int\Big[ \frac{\mathbf{a}_{1}\times\mathbf{a}_{2}}{|\mathbf{a}_{1}|^{3}}\frac{1}{|\zeta|}+O(1)\Big]d\zeta.
	\end{equation}
	
	Arms and Hama then consider integration over the limit $ \epsilon\leq|\zeta|<1 $, one obtains that \[\mathbf{q}_{ij} =\frac{k}{2\pi} \frac{\mathbf{a}_{1}\times\mathbf{a}_{2}}{|\mathbf{a}_{1}|^{3}}\frac{1}{|\epsilon|}+O(1)\Big] \] and by re-substituting $ \mathbf{a}_{n} = \frac{1}{n!}\partial^{n}\mathbf{r}_{ij}/\partial\zeta^{n} $ they obtain:
	\begin{equation}
		\frac{4\pi}{k}\mathbf{q}_{ij} = \frac{(\partial\mathbf{r}/\partial s)_{i}\times(\partial^{2}\mathbf{r}/\partial s^{2})_{i}}{|(\partial\mathbf{r}/\partial s)_{i}|^{3}}\log(\frac{1}{\epsilon})+O(1)
	\end{equation}
	
	Then considering the infinitesimal limit for $ \epsilon<<1 $ and ignoring the terms $ O(1) $, whch is equivalent to ignoring long distance effects, we may write the previous expression as \begin{equation}
		\frac{\partial\mathbf{r}}{\partial t} = \frac{(\partial\mathbf{r}/\partial s)\times(\partial^{2}\mathbf{r}/\partial s^{2})}{|(\partial\mathbf{r}/\partial s)|^{3}}
	\end{equation}
	
	The high nonlinearity of this equation makes for explicit solutions to be hard to find. We then look for self-similar solutions to understand the behavior of vortices under this equation.
	\subsection{Properties of Vortex Filament Equation}
	
		We present in this section some of the basic properties of VFE and conclude it with a proof that the only self-similar translating solutions that lie in a plane for all time are circles moving in the binormal direction.
	\begin{lemma}[Arc-Length Commutator]
		Let $\gamma$ be a curve evolving by the Vortex Filament Equation. Then it satisfies the following commutator relation, where $x$ is arclength:
		\begin{equation}
		\Large{\Big[\frac{\partial}{\partial t}, \frac{\partial}{\partial s}\Big]}
		\normalsize :=
		\Large \frac{\partial^{2}}{\partial t \partial s}-\frac{\partial^{2}}{\partial s \partial t} \normalsize =0.
		\end{equation}
		In other words,
		\begin{align}
		\gamma_{ts}=\gamma_{st}
		\end{align}
	\end{lemma}
	
	\begin{proof}
		Using $\frac{\partial}{\partial s}=\lvert\frac{\partial\gamma}{\partial u}\rvert^{-1}\frac{\partial}{\partial u}$, where $u$ is an arbitrary parameter, we compute
		\begin{align}
		\gamma_{st} &= (\lvert\gamma_{u}\rvert^{-1}\gamma_{u})_{t}\\
		&=-\lvert\gamma_{u}\rvert^{-3}\langle\gamma_{tu},\gamma_{u}\rangle\gamma_{u}+\lvert\gamma_{u}\rvert^{-1}\gamma_{tu}\\
		&=-\lvert\gamma_{u}\rvert^{-3}\langle\kappa B_{u},\gamma_{u}\rangle\gamma_{u}+\gamma_{ts}\\
		&=\kappa\langle B_{s},T\rangle\gamma_{s}+\gamma_{ts}\\
		&=\gamma_{ts} ,
		\end{align}
		because $B_{s}=-\tau N$.$\qed$
	\end{proof}
		
	\begin{corollary}[Evolution of the Normal Vector, Curvature and Torsion]
		Let $\gamma$ be a curve evolving under Binormal Flow, $N$ be its corresponding normal vector and $ F(s,t) = \tau^{2}-\frac{\kappa_{ss}}{\kappa} $. Then the evolution of the normal vector satisfies $\partial_{t}N = \tau\kappa T-F(s,t)B$, the evolution for curvature satisfies $  \kappa_{t} = -(2\kappa_{s}\tau+\tau_{s}\kappa)$, and the evolution equation for torsion satisfies $\partial_{t}\tau = -\kappa\kappa_{s}+F_{s}(s,t)$.	
	\end{corollary}
	
	\begin{proof}
		We have that $\lVert N\rVert^{2} = 1$ so then $\partial_{t}(\lVert N\rVert^{2}) = 2\langle N_{t},N\rangle = 0$. We can therefore separate $\partial_{t}N$ into tangential and binormal components. We now define $F(x,t) = \langle N,B_{t}\rangle$ and calculate
		\begin{align}
			\frac{d}{dt}\langle N,T\rangle &= \langle N_{t},T\rangle+\langle N,T_{t}\rangle=0\\
			\langle N_{t},T\rangle &=-\langle N,T_{t}\rangle=\tau\kappa\\
			\langle N_{t},B\rangle &=-\langle N,B_{t}\rangle = -F(s,t).
		\end{align}
		Since by the definition of VFE, $T_{t}=\partial_{s}(\kappa B) = \kappa_{s}B-\tau\kappa N$. Using that $\langle N,T\rangle = \langle N,B\rangle = 0$, we compute
		\begin{align}
			N_{t} &= \Big[\frac{\partial_{ss}\gamma}{\kappa}\Big]_{t}\\
			&=-\frac{\kappa_{t}}{\kappa^{2}}\gamma_{ss}+\frac{1}{\kappa}\gamma_{sst}.
		\end{align}
		We now calculate $ \gamma_{sst} $ as follows:\begin{align}
			\partial_{s}\gamma &= \frac{1}{|\gamma_{x}|}\gamma_{x}\\
			\partial_{ss}\gamma &= \frac{1}{|\gamma_{x}|}\partial_{x}(|\gamma_{x}|^{-1}\gamma_{x})\\
			&= \frac{\gamma_{xx}}{|\gamma_{x}|^{2}}\\
			\partial_{sst}\gamma &=\partial_{t}(\frac{\gamma_{xx}}{|\gamma_{x}|^{2}})\\
			&= -2|\gamma_{x}|^{-3}\langle\gamma_{xt},\gamma_{x}\rangle\gamma_{xx}+ \frac{\gamma_{xxt}}{|\gamma_{x}|^{2}}\\
			&= \frac{\gamma_{txx}}{|\gamma_{x}|^{2}} = (\kappa B)_{ss} = \kappa_{ss}B+2\kappa_{s}B_{s}+\kappa B_{ss}\\
			&= \kappa_{ss}B-2\kappa_{s}\tau N+\kappa(-\tau_{s}N+\kappa\tau T-\tau^{2}B).
		\end{align} 
		
		Combining our results we get that \begin{align}
			N_{t} &= -\frac{\kappa_{t}}{\kappa}N+\frac{\kappa_{ss}}{\kappa}B-2\frac{\kappa_{s}\tau}{\kappa}N-\tau_{s}N+\kappa\tau T-\tau^{2}B.\\
			&= -\Big(\frac{\kappa_{t}}{\kappa}+2\frac{\kappa_{s}\tau}{\kappa}+\tau_{s}\Big)N+\tau\kappa T+\Big(\frac{\kappa_{ss}}{\kappa}-\tau^{2}\Big)B.
		\end{align}
		From this we can conclude that $ \kappa_{t} = -(2\kappa_{s}\tau+\tau_{s}\kappa) $ since $ \langle N_{t},N\rangle  = 0$, and that $ F(s,t) = \tau^{2}-\frac{\kappa_{ss}}{\kappa} $.
		
		 Our final expression for the evolution of the normal vector is then \begin{equation}
			N_{t} = \tau\kappa T+\Big(\frac{\kappa_{ss}}{\kappa}-\tau^{2}\Big)B.
		\end{equation}
		
		To compute the evolution equation for torsion we first compute the evolution for the Binormal Vector:
		\begin{align}
			B &= T\times N\\
			B_{t}&=T_{t}\times N+T\times N_{t}\\
			&=-\kappa_{s}T+F(s,t)N = -\kappa_{s}T+\Big(\tau^{2}-\frac{\kappa_{ss}}{\kappa}\Big)N.
		\end{align}
		Continuing, we calculate
		\begin{align}
			\partial_{t}(\tau N) &=\partial_{t}(\partial_{s}B)\\
			&= \partial_{s}(\partial_{t}B)\\
			&= \partial_{s}\Big[-\kappa_{s}T+F(s,t)N\Big]\\
			&= -\kappa_{ss}T-\kappa_{s}\kappa N+F_{s}(s,t)N-\kappa F(s,t) T+\tau F(s,t) B\\
			&= \Big[-\kappa_{ss}-\kappa F(s,t)\Big]T - \Big[\kappa_{s}\kappa+F_{s}(s,t)\Big]N+\tau F(s,t)B.
		\end{align}
		Thus since $\partial_{t}(\tau N) = \partial_{t}\tau N+\tau N_{t}$ we have that \[ \partial_{t}\tau= -\kappa\kappa_{s}+2\tau\tau_{s}-\frac{\kappa_{sss}}{\kappa}+\frac{\kappa_{ss}\kappa_{s}}{\kappa^{2}}. \]
	\end{proof}
	
	\begin{lemma}[Planar Translating Self-similar Solutions]
		The only solutions of VFE that lie in a translating plane are the line and the circle, with the latter traveling in the binormal direction.
		
	\end{lemma}
	\begin{proof}
		Take $\gamma = (f(s,t),g(s,t),h(s,t))$ to be a space curve in $\mathbb{R}^{3}$ parametrized by arclength and having torsion 0.
		
		Then $\partial_{t}\gamma = (\partial_{t}f(s,t),\partial_{t}g(s,t),\partial_{t}h(s,t))$, $\partial_{s}\gamma = (\partial_{s}f(s,t),\partial_{s}g(s,t),\partial_{s}h(s,t))$, and $\partial_{ss}\gamma = ( \partial_{ss}f(s,t),\partial_{ss}g(s,t),\partial_{ss}h(s,t))$.
		
		Using the definition of VFE equation we calculate that
		\begin{align}
		\partial_{t}\gamma&= \kappa B\\	
		&= \partial_{s}\gamma\times\partial_{ss}\gamma \\
		&= ( g_{s}h_{ss}-g_{ss}h_{s},f_{ss}h_{s}-f_{s}h_{ss}, f_{s}g_{ss}-f_{ss}g_{s})
		\end{align}
		
		Since curvature is invariant under rotation and the solutions of the curve are invariant under translations, it suffices to solve for equations lying in the x-y plane of our chosen coordinate system.
		
		Also the Frenet Equations under the initial conditions become	
		\begin{align}
		\partial_{s}T &= \kappa N\\
		\partial_{s}N &= -\kappa T\\
		\partial_{s}B &= 0		
		\end{align}
		From (104) we can conclude that $ B(s,t) =B(0,t) $. Similarly $\kappa(x,t) = (f_{s}g_{ss}-g_{ss}f_{s})(s,t)$ since the curve lies on the x-y plane, and by Corollary 3.1.2 we have that $ \kappa_{t} = 0 $. Hence  \[ \kappa(s,t) = \kappa(s,0) = (f_{s}g_{ss}-g_{ss}f_{s})(s,0). \]
		
		Then, since for all $ t $ we have that \[ \partial _{t}h(s,t) = (f_{s}g_{ss}-g_{s}f_{ss})(s,t)=\kappa(s,t), \] the previous argument gives that $ \partial_{t}h(s,t) = \kappa(s,0) $. Integrating this last expression gives us that \[ h(s,t) = \kappa(s,0)t. \]
		
		So then since $\lVert B(t)\rVert = 1$ we have that $\lVert\partial_{t}\gamma(s,t)\rVert = \lvert\kappa^{2}\rvert$, and this can be rewritten as $f_{t}^{2}+g_{t}^{2}+h_{t}^{2} =f_{t}^{2}+g_{t}^{2}+\kappa^{2}= \kappa^{2}$ and this gives that $f_{t}^{2} = g_{t}^{2} = 0$.
		
		This implies that $f(s,t) = w(s)$ and $g(s,t) = z(s)$ completely determine the solutions to the binormal flow equation with $\tau = 0$ so long as $w(x)$ and $z(s)$ satisfy	
		\begin{align}		
		w'(s)^{2}+z'(s)^{2}+\kappa_{s}^{2}t^{2} &= 1
		\end{align}
		for all $s$ and $t$, where $\kappa = w'z''-z'w''$, due to the parametrization by arc length. Since this must hold for all $t \in[0,\infty)$ it forces $\kappa_{s} = 0$
		
		Putting this all together we have that the equations $w(s)$ and $z(s)$ determine a solution to the binormal equation with $\tau = 0$ so long as they solve the system of differential equations:
		\begin{align}
		w'(s)^{2}+z'(s)^{2} &= 1\\
		w'z''-z'w'' &= C
		\end{align}
		where $C$ is a constant.
		
		The only solutions of these equations depend on the value assigned to $C$. If $C$ is 0, then it follows that $w(s) = as$ and $z(s) = \sqrt{1-a^{2}}s$, and if $C\neq0$ then the solutions are circles with a radius of $1/C$ since $\tau = 0$ and curvature is constant.$\qed$
	\end{proof}
	\subsection{The Hasimoto Transform and Hasimoto's Explicit Solution}
	We now present the results obtained by Hasimoto on his paper on VFE and his explicit solution for a traveling wave soliton ([15]). The most striking part of his paper is the Hasimoto Transform which manages to turn VFE into a form that satisfies a Nonlinear Cubic Schroedinger Partial Differential Equation (NLCSE). This then allows one to find solutions to VFE by finding solutions to the fully integrable NLCSE. It was through this process that Hasimoto found one of the first explicit soliton solutions to VFE. We now present the results from his paper.
	
	\begin{theorem}[Hasimoto Transform $ \text{[15]} $]
		Let $ \Gamma:\mathbb{R}^{3}\times\mathbb{R}\rightarrow\mathbb{R}^{3} $ be a differentiable curve parametrized by arclength of curvature $ \kappa $ and torsion $ \tau $ that satisfies the Vortex Filament Equation. If we take the transformation $ \psi = \kappa \exp\Big(i\int_{0}^{s}\tau ds\Big) $ then $ \psi $ will satisfy the following NLCSE:\[ \frac{1}{i}\partial_{t}\psi = \partial_{ss}\psi+\frac{1}{2}(|\psi|^{2}+A)\psi. \]
	\end{theorem}
		
	\begin{proof}
		The Frenet system of coordinates gives the following equations \begin{align}
			\Gamma_{s} &= T\\
			T_{s} &= \kappa N\\
			N_{s} &= \tau B	-\kappa T\\
			B_{s} &= -\tau N.	
		\end{align}
		
		Combining the last two equations in the following form gives \begin{equation}
			(N+iB)_{s}=-i\tau(N+iB)-\kappa T
		\end{equation}
		which leads to the introduction of the new variables $ \mathfrak{N} $ and $ \psi $ defined as \begin{align}
		\mathfrak{N} &= (N+iB)\exp\Big(i\int_{0}^{s}\tau ds\Big)\\
		\psi &= \kappa\exp\Big(i\int_{0}^{s}\tau ds\Big).
		\end{align}
		
		From the Frenet equations we obtain the following expressions, where the last one follows from the defintion of VFE and use of (66) and (67): \begin{align}
			\mathfrak{N}_{s} &= -\psi T\\
			T_{s}&= \text{Re}\big[\psi\mathfrak{N}\big] = \frac{1}{2}\Big(\bar{\psi}\mathfrak{N}+\psi\bar{\mathfrak{N}}\Big)\\
			T_{t} &= \text{Re}\big[i\psi'\bar{\mathfrak{N}} \big] = \frac{1}{2}i\big(\psi'\bar{\mathfrak{N}}-\bar{\psi}'\mathfrak{N}\big).
		\end{align}
		 
		The following relations shows that this creates an orthogonal system of equations \begin{equation*}
			T\cdot T = 1,\qquad\mathfrak{N}\cdot\bar{\mathfrak{N}} = 2,\qquad\mathfrak{N}\cdot\mathfrak{N} = 0,\qquad \mathfrak{N}\cdot T = 0,\quad \text{etc.}
		\end{equation*}
		
		Now we derive the evolution equation for $ \mathfrak{N} $ by expressing it in this new orthogonal system as\begin{equation}
			\mathfrak{N}_{t} = \alpha\mathfrak{N}+\beta\bar{\mathfrak{N}}+\gamma T.
		\end{equation}
		
		We now determine the values of the coefficients as follows \begin{align}
			\alpha+\bar{\alpha} &=\frac{1}{2}\big(\bar{\mathfrak{N}}_{t}\cdot\mathfrak{N}+\mathfrak{N}\cdot\bar{\mathfrak{N}}_{t}\big)= \frac{1}{2}\partial_{t}(\bar{\mathfrak{N}}\cdot\mathfrak{N}) = 0,\qquad \alpha = iR\\
			\beta &= \frac{1}{2}\mathfrak{N}\cdot\bar{\mathfrak{N}} =\frac{1}{4}\partial_{t}\big(\mathfrak{N}\cdot\mathfrak{N}\big) = 0\\
			\gamma&=-\mathfrak{N}\cdot T_{t} = -i\psi
		\end{align}
		where $ R $ is a real function. This then simplifies to \begin{equation}
			\mathfrak{N}_{t} = i\big(R\mathfrak{N}-\psi T\big).
		\end{equation}
		
		We now take the time derivative of (68) and the arclength derivative of (75) obtaining \begin{align}
			\mathfrak{N}_{st} &=-\psi T-\psi T_{t}=-\psi T-\frac{1}{2}i\psi\big(\psi_{s}\bar{\mathfrak{N}}-\bar{\psi}_{s}\mathfrak{N}\big)\\
			\mathfrak{N}_{ts} &= i\big[R_{s}\mathfrak{N}-R\psi T-\psi_{ss}T-\frac{1}{2}\psi_{s}\big(\bar{\psi}\mathfrak{N}+\psi\bar{\mathfrak{N}}\big)\big].
		\end{align}
		
		We can now equate the coefficients of $ T $ and $ i\mathfrak{N} $ giving us \begin{align}
			-\psi = -i\big(\psi_{ss}+R\psi\big)\\
			\frac{1}{2}|\psi|^{2} = R_{s}-\frac{1}{2}\psi_{s}\bar{\psi}.
		\end{align}
	Solving this last equation for $ R $ gives that $ R = \frac{1}{2}\big(|\psi|^{2}+A\big) $ which reduces the first equation down to \begin{equation}
		\frac{1}{i}\partial_{t}\psi = \partial_{ss}\psi+\frac{1}{2}(|\psi|^{2}+A)\psi.
	\end{equation}
	\end{proof}
	
	\begin{theorem}[Hasimoto's Traveling Wave  $ \text{[15]} $]
		Consider a soliton solution to NLCSE defined by Theorem 3.2.1 such that $ \kappa = 0 $ as $ s\rightarrow\infty $. This then gives as a solution a traveling wave with a kink that becomes a line at infinity.
	\end{theorem}
	\begin{proof}
	 First we introduce the new variable $ \zeta = s-ct $ where $ c $ can be taken to be the 'speed' of the translation. This variable is introduced in order to be able to work with a soliton of the Schroedinger equation derived in the previous theorem. Using this variable then gives \begin{equation}
	 	\psi = \kappa(\zeta)\exp\Big[i\int_{0}^{s}\tau(\zeta)ds\Big].
	 \end{equation}
	 
	 We plug this variable into the Schroedinger equation yielding the following real and imaginary parts, respectively, \begin{align}
	 	\-c\kappa[\tau(\zeta)(-ct)] &= \kappa''-\kappa\tau^{2}+\frac{1}{2}(\kappa^{2}+A)\kappa\\
	 	c\kappa' &= 2\kappa'\tau+\kappa\tau'
	 \end{align}
	 and integrating the last equation gives \begin{equation}
	 	(c-2\tau)\kappa^{3} = 0
	 \end{equation}
	 where we have used our curvature limit to determine the constant of integration. By (129) we then have \[
	 \tau =\tau_{0}= \frac{1}{2}c = \text{constant}
	 \]
	 assuming that curvature is not identically zero. Using this we can integrate (130)
	 \[ \kappa = 2\nu \sech(\nu\zeta) \]
	 so long as $ A $ is a constant determined by $ A = 2(\tau_{0}^{2}-\nu^{2}) $. Now that the torsion and curvature are determined we can get the actual shape of the filament by substituting these into our original Frenet frame of reference and solving the following equation for the binormal vector:\begin{align}
	 	\tau_{0}(T'-\kappa N) = \big[(1/\kappa)\big(B''+\tau_{0}^{2}B\big)\big]'+\kappa B' &= 0\\
	 	\frac{d^{3}}{d\eta^{3}}B+\tanh(\eta)\frac{d^{2}}{d\eta^{2}}B+\big(S^{2}+\sech^{2}(\eta)\big)\frac{d}{d\eta}B+S^{2}\tanh(\eta)B &= 0
	 \end{align}
	 where we have made the substitutions $ \eta = \nu\zeta $ and $ S = \frac{\tau_{0}}{\nu} $.
	 
	 A solution to this equation can then be obtained if we note that \[ \mathbf{C} = \frac{dB}{d\eta}+\tanh(\eta)B \] is a solution of the equation \[ d^{2}\mathbf{C}/d\eta^{2}+\big(S^{2}+2\sech^{2}(\eta)\big)\mathbf{C} = 0 \] which has as solutions \begin{equation}
	 	\quad(1-S^{2}\mp2iS\tanh(\eta))e^{\pm iS\eta}.
	 \end{equation}
	 
	 This finally gives us the binormal vector\begin{equation}
	 	B = \sech(\eta),\quad(1-S^{2}\mp2iS\tanh(\eta))e^{\pm iS\eta}.
	 \end{equation}
	 which we can then substitute into our equations of the Frenet Frame and using the assumption, without loss of generality, that the filament is parallel to the x-axis at infinity: \begin{align}
		 T_{x}\rightarrow1\quad\text{as}\quad\eta\rightarrow\infty\\
		 N_{y}+iN_{z} = -i\big(B_{y}+iB_{z}\big) = e^{i(\tau_{0}\zeta+\sigma(t))}
	 \end{align}
	 Here $ \sigma(t) $ is a real function of $ t $ and the subscripts denote the component of the vector in the x, y, or z axis.
	 
	 Finally straightforward calculation gives the final expressions 
	 \begin{align}
	 	\begin{cases}
	 	\Gamma:& x = s-\frac{2\mu}{\nu}\tanh(\eta), \quad y+iz = re^{i\Theta},\\
	 	T:& T_{x} = 1-2\mu\sech^{2}(\eta),\quad T_{y}+iT_{z} =-\nu r(\tanh(\eta)-iS)e^{i\Theta}\\
	 	N:& N_{x} = 2\mu\sech^{2}(\eta)\sinh(\eta),\quad N_{y}+iN_{z} = -[1-2\mu(\tanh(\eta)-iS)\tanh(\eta)]e^{i\Theta}\\
	 	B:& B_{x} = 2\mu S\sech(\eta), B_{y}+iB_{z} = i\mu(1-S^{2}-2iS\tanh(\eta))e^{i\Theta}
	 	\end{cases}
	 \end{align}
	 where \begin{align}
	 	\mu &= \frac{1}{1+T^2} = \frac{\nu^{2}}{\nu^{2}+\tau_{0}^{2}},\quad r =\frac{2\mu}{\nu}\sech(\eta)\\
	 	\eta &= \nu\zeta = \nu(s-2\tau_{0}t), \quad\Theta = S\eta+\nu^{2}(1+S^{2})t = \tau_{0}s+(\nu^{2}-\tau_{0}^{2})t.
	 \end{align}
	 
	 This then provides the traveling wave soliton solution that was first described by Hasimoto.
	\end{proof}
	
	\begin{remark}
		Another way one can arrive to the Hasimoto Transform and solution is by assuming the y and z components to be dependent on the x component. Using Taylor expansions of the norm of the curve where second order terms are thrown out to simplify the system of equations obtained, one can then substitute $ \Psi \equiv -(y+iz) $ to find that this $ \Psi $ satisfies the NLCSE. If one then looks for a soliton solution of this equation that is slowly varying one finds Hasimoto's traveling wave solution. 
	\end{remark}
	\subsection{Self Similar Dilating Solutions of the Vortex Filament Equation}
	In this section we present a summary the results of Banica and Vega on the discovery of self-similar dilating solutions of curves with a corner ([16],). They developed a set of solutions that only just fail to exist  in $ H^{3/2} $ and develop a corner in finite time. We present their general results as well as present a rough sketch of their proof, but we suggest further reading the works of Banica and Vega to illustrate the depth of the problem in finding such solutions and some of the applications of VFE to scattering theory for Schroedinger equations. 
	
	First it is necessary to find a way to translate back from the tangent and normal vectors a filament that satisfies VFE where curvature is allowed to become 0. Banica and Vega overcome this difficulty by creating another frame of reference $ (T, e_{1}, e _{2}) $  governed by \[
	\begin{bmatrix}
	T\\
	e_{1}\\
	e_{2}
	\end{bmatrix}_{x} = \begin{bmatrix}
	0 &\alpha &\beta\\
	-\alpha &0 &0\\
	-\beta &0 &0
	\end{bmatrix}\begin{bmatrix}
	T\\
	e_{1}\\
	e_{2}
	\end{bmatrix}.
	\]
	This is in turn a reformulation of the Hasimoto Transform since from this coordinate frame change we can define \[
	\psi(t,x) = \alpha(x,t)+i\beta(x,t)
	\] and find that this solves the NLCSE with $ A(t) = \alpha^{2}(t,0)+\beta^{2}(t,0) $. Then we can define $ N = e_{1}+ie_{2}$ and find that this is in turn equivalent to the $ \mathfrak{N} $ from section 3.2, showing that the Hasimoto transform is one way that one could obtain this new frame of reference.
	
	These vectors can then be combined to yield the evolution equations \begin{equation*}
	T_{x} = \text{Re}(\bar{\psi}N),\quad N_{x} = -\psi T, \quad T_{t} =\text{Im}(\bar{\psi}N), \quad N_{t} = -i\psi_{x}T+i(|\psi|^{2}-A(t))N
	\end{equation*}
	which can then be used to construct these vectors with the aid of imposing $ (T,N) = (e_{0},e_{1}+ie_{2}) $. Banica and Vega then define \begin{equation}
	\gamma(t,x) := P+\int_{t}^{t_{0}}(T\wedge T_{xx})(\tau, x_{0})d\tau+\int_{x}^{x_{0}}T(t,s)ds
	\end{equation}
	and find that this solves the VFE.
	
	The focus is first on self-similar dilating solutions of VFE \[ \gamma(t,x) = \sqrt{t}G\Big(\frac{x}{\sqrt{t}}\Big). \] [21] showed that the family of self-similar solutions $ \{\gamma_{a}\}_{a\in\mathbf{R}^{+}} $ is characterized by explicit curvature and torsion $ c_{a} = \frac{a}{\sqrt{t}},\quad\tau_{a} = \frac{x}{2t} $. These self-similar solutions with a corner have been shown to be analogous with the 'delta-wing' vortex and serves as a good analogy of its behavior.
	
	Using this we then look at perturbation solutions using Hasimoto's transform using the previous explicit torsion and curvature, giving the family of filament functions \begin{equation}
	\psi_{a}(t,x) = \frac{a}{\sqrt{t}}e^{ix^{2}/4t}
	\end{equation}
	which solve the NLCSE\[ i\psi_{t}+\psi_{xx}+\frac{1}{2}\Big(|\psi|^{2}-\frac{a^2}{t}\Big)\psi = 0. \] The corner of the soliton corresponds to the initial condition $ \psi_{a}(0,x) = a\delta_{x = 0} $, where $ \delta_{x = 0} $ is the delta distribution centered at $ x =0 $.
	
	Here we shall present verbatim the theorems Vega and Banica derived and afterwards we present their discussion on the results ([16]), but first we must give some definitions to provide sufficient background.
	
	\begin{definition}[Filament Function]
		We shall define as a filament function the function $ \psi $ that is obtained when carrying a solution of VFE through a transform, such as the Hasimoto transform, to create an equation that solves the NLCSE, i.e if it has curvature $ c $ and torsion $ \tau $ then the corresponding vortex filament function is \[ \psi = c(x,t)e^{-\int_{0}^{x}\tau ds} \]
	\end{definition}
	
	\begin{definition}[Weighted Space]
		We define the weighted space $ X^{\alpha}, $\[ X^{\alpha}:= \{f\in L^{2}|\zeta^{\alpha}\hat{f}(\zeta)\in L^{\infty}(|\zeta|\leq 1)  \}\]	\end{definition}
	
	\begin{theorem}[Continuation of Vortex after singularity $ \text{[16]} $]
		Let $ \gamma(1) $ be a perturbation of a self-similar solution $ \gamma_{a} $ at time $ t = 1 $ in the sense that the filament function of $ \gamma(1) $ is $ (a+u(1,x))e^{ix^{2}/4} $, with $ \partial^{k}_{x}u(1) $ small in $ X^{\alpha} $ with respect to $ a $ for all $ 0 \leq k\leq4  $, for some $ \alpha<\frac{1}{2} $.
		
		One can construct a solution $ \gamma\in C([-1,1], Lip)\cap C([-1,1]\backslash \{0\},C^{4}) $ for VFE on $ t\in[-1,1]\backslash\{0\} $ with a weak solution on the entire interval [-1,1]. The solution $ \gamma $ is unique on the subset of $ C([-1,1], Lip)\cap C([-1,1]\backslash \{0\},C^{4}) $ such that the associated filament functions at times $ \pm 1 $ can be written as $ (a+u(\pm1,x))e^{ix^{2}/4} $ with $ \partial^{k}_{x}u(\pm1) $ small in $ X^{\alpha} $ with respect to a for all $ 0\leq k\leq4 $.
		
		This solution enjoys the following properties:
		\begin{itemize}
			\item there exists a limit of $ \gamma(t,x) $ and of its tangent vector $ T(x,t) $ at time zero, and \begin{equation}
			\sup_{s}|\gamma(t,x)-\gamma(0,x)|\leq C\sqrt{|t|},\quad \sup_{|x|\geq\epsilon>0}|T(x,t)-T(0,x)|\leq C_{\epsilon}|t|^{\frac{1}{6}^{-}}
			\end{equation}
			\item $ \forall t_{1},t_{2}\in[-1,1]\backslash \{0\} $ the following asymptotic properties hold\begin{equation}
			\gamma(t_{1},x)-\gamma(t_{2},x) =O(\frac{1}{x}),\quad T(t_{1},x)-T(t_{2},x) = O(\frac{1}{x})
			\end{equation}
			\item $ \exists T^{\infty}\in \mathbb{S}^{2}, N^{\infty}\in \mathbb{C}^{3} $ such that uniformly in $ -1\leq t\leq1 $\begin{equation}
			T(t,x)-T^{\infty} = O\Big(\frac{1}{\sqrt{x}}\Big),\quad (N+iB)(t,x)-N^{\infty}e^{ia^{2}\log \frac{\sqrt{t}}{x}-ix^{2}/4t}=O(\frac{1}{\sqrt{x}}),
			\end{equation}
			\item modulo rotation and a translation, we recover at the singularity point (0,0) the same structure as for $ \gamma_{a} $:\begin{equation}
			\lim_{x\rightarrow 0^{\pm}}T(0,x) = A^{\pm}_{a},\quad \lim_{x\rightarrow 0^{\pm}}\lim_{t\rightarrow 0}(N+iB)(t,x)e^{-ia^{2}log \frac{\sqrt{t}}{x}-ix^{2}/4t}=B^{\pm}_{a}.
			\end{equation}
		\end{itemize}
		
		Where $ u $ in the previous proof is the solution to the following PDE \[ iu_{t}+u_{xx}+\frac{1}{2t}(|u+a|^{2}-a^{2})(u+a) = 0 \] which arises by taking perturbations of the filament function $ \psi_{\alpha} $ and proceeding to analyze the long term behavior of $ u $.
		\end{theorem}
		
		\begin{theorem}[VFE with values with a corner $ \text{[16]} $]
			Let $ \gamma_{0} $ be a smooth $ C^{4} $ curve, except at $ \gamma_{0}(0) = 0 $ where a corner is located, i.e. that there exist $ A^{+} $ and $ A^{-} $ two distinct non-colinear unitary vectors in $ \mathbf{R}^{3} $ such that\begin{equation}
			\gamma_{0}'(0^{+}) =A^{+},\quad\gamma_{0}'(0^{-}) = A^{-}.
			\end{equation}
			We set a to be the real number given by the unique self-similar solution of VFE with the same corner as $ \gamma_{0} $ at time $ t=0 $. We suppose that the curvature of $ \gamma_{0}(x) $ (for $ x \neq0 $) satisfies $ (1+|x^{4}|) c(x)\in L^{2}$ and $ |x|^{\zeta}c(x)\in L^{\infty}_{(|x|\leq1)} $, small with respect to a.
			
			Then there exists $ \gamma(t,x)\in C([-1,1], Lip)\cap C([-1,1]\backslash\{0\},C^{4})
			,$ regular solution of the VFE for $ t\in[-1,1]\backslash\{0\}, $ having $ \gamma_{0} $ as value at time $ t=0 $, and enjoying all properties from the previous theorem.
		\end{theorem}

	\begin{remark}
		The first theorem constructs a solution for VFE that although presenting a discontinuity at time $ t=0 $ and show some convergence in large values in space and for small values in time, and shows that this discontinuity is in fact a corner singularity at the point $ (0,0) $. The second theorem applies these results to the evolution of curves with this corner singularity, and the proof of this theorem occurs naturally in the proof of 3.3.4.
		
		Since we do not present explicit proofs of these results in these sections we direct the reader to Vega and Banica's papers [17], [18], [19], [20], and [21] which provide a rich exposition of the underlying structure of the equation as well as the spaces in which solutions to VFE exist.
	\end{remark}
	\subsection{Self-Similar Rotating Solutions of Vortex Filament Equation}
		In this section we derive expressions of rotating self-similar solutions of VFE in 3 dimensional space as well as those that rotate around a singular plane. We begin first by deriving those rotating in 3-space, and those in a singular plane follow as a corollary. 
		\begin{theorem}
			A derivation for the ODEs governing the motion of self-similar rotating solutions of the Vortex Filament Equation. We find, without loss of generality, that considering a self-similar curve in $ \mathbf{R}^3 $ of the form $ \Gamma(x) = (x, y(x), z(x)) $ that is rotating in space. Let $ \Gamma^{x}, \Gamma^{y}, \Gamma^{z} $ denote the components in the x, y, and z axes so as to not confuse them with the notation for partial derivatives. We obtain that rotating self-similar solutions around a single axis can be split into the following 2 cases:
			
			(1) Rotation around the z or y axis: This is the case for rotation around the y-axis, for rotation around the z-axis simply switch the values for $ \Gamma^z $ and $ \Gamma^y $. In this case we find the following system of equations describing the vector $ \Gamma $:
			\begin{align}
				&\Gamma^{x} = x\\
				&\Gamma^{y} = C_{1}x\\
				&\Gamma^{z} = z(x)
			\end{align}
			where $ z(x) $ must be a solution of the nonlinear ODE\begin{equation}
				 -\frac{z'(x)}{\sqrt{1+C_{1}^{2}+z'(x)^{2}}}=\frac{1}{2}(1+C_{1}^{2})(x^{2}+2C_{2}).
			\end{equation}
			This gives a two parameter family of solutions for rotations of this type.
			
			(2) Rotation around the x-axis: For this case the components of the curve in this situation are as follows\begin{align}
				&\Gamma^{x} = x\\
				&\Gamma^{y} = \lambda z(x)\\
				&\Gamma^{z} = z(x)
			\end{align}
			where $ \lambda \in \mathbf{R}$ and $ z(x) $ solves the following ODE \begin{equation}
				z'(x) = \pm\sqrt{ \frac{4-(1+\lambda^2)^2(z^2 + 2C_1)^2}{(1+\lambda^2)^3(z^2 + 2C_1)^2}}
			\end{equation}
			and with this equation we have a 3 parameter family of solutions for rotation around the x-axis.
		\end{theorem}
		
		\begin{proof}
			Consider a solution of VFE of the form $\Gamma (x,t)=e^{tM}\Gamma_{0}(x)$, where $\Gamma_{0}$ is a curve in $\mathbb{R}^{3}$ and $M$ is the following skew-symmetric matrix:
			$M=
			\begin{Bmatrix}
			0			&-\omega_{y} &\omega_{z}\\
			
			\omega_{y} &0 &-\omega_{x}\\
			
			-\omega_{z} &\omega_{x} &0
			\end{Bmatrix}$.
			Here the entries are scalars representing the angular velocities of each axis' rotation.
			
			Applying the equations of VFE:
			\begin{align}
			\Gamma_{t} &= \Gamma_{s}\times\Gamma_{ss}\\
			Me^{tM}\Gamma_{0} &= e^{tM}(\Gamma_{0}'\times\Gamma_{0}'')\\
			M\Gamma_{0} &= \Gamma_{0}'\times\Gamma_{0}''
			\end{align}
			
			Suppose $\Gamma_{0} = \langle x, y(x), z(x)\rangle$. Taking the derivate with respect to arc-length of this curve gives $\frac{d \Gamma_{0}}{ds} = \frac{1}{\sqrt{1+y_{x}^{2}+z_{x}^{2}}}\frac{d\Gamma_{0}}{dx}$.
			
			So then we give expressions for the first and second order derivatives with respect to arclength:
			\begin{align}
			\frac{d\Gamma_{0}}{ds} &= [1+y_{x}^{2}+z_{x}^{2}]^{-1/2}\langle 1,y'(x),z'(x)\rangle\\
			\frac{d^{2}\Gamma_{0}}{ds^{2}} &= 	[1+y_{x}^{2}+z_{x}^{2}]^{-1/2}(\frac{d}{dx})(\frac{d\Gamma_{0}}{ds})\\
			\frac{d^{2}\Gamma_{0}}{ds^{2}} &= [1+y_{x}^{2}+z_{x}^{2}]^{-2}\{-(y''y'+z''z')\langle 1, y',z'\rangle+[1+y_{x}^{2}+z_{x}^{2}]\langle 0,y'',z''\rangle\}
			\end{align}
			Plugging our results into (161) and using equations (162)-(164), as well as the definition of $M$ we have the following system of ODE's describing the behavior of solutions:
			\begin{align}
			\omega_{z}z-\omega_{y}y &=(y'(x)z''(x)-z'(x)y''(x))[1+y_{x}^{2}+z_{x}^{2}]^{-3/2}\\
			\omega_{y}x-\omega_{x}z &= -z''(x)[1+y_{x}^{2}+z_{x}^{2}]^{-3/2}\\
			\omega_{x}y-\omega_{z}x &= y''(x)[1+y_{x}^{2}+z_{x}^{2}]^{-3/2}
			\end{align}
			
			\textbf{Rotation around the y-axis}:
			In this situation the constants $\omega_{x} =\omega_{z} = 0$ and $\omega_{y} =1 $, where the latter is taken for simplicity. Then the equations (insert numbers here) become:
			\begin{align}
				-y &=(y'(x)z''(x)-z'(x)y''(x))[1+y_{x}^{2}+z_{x}^{2}]^{-3/2}\\
				x &=-z''(x)[[1+y_{x}^{2}+z_{x}^{2}]^{-3/2}\\
				0 &= y''(x)[1+y_{x}^{2}+z_{x}^{2}]^{-3/2}	
			\end{align} 
			so that by the last equation we have that $ y(x) = C_{1}x $. This then transforms (165) and (166) into the second order nonlinear differential equation:\begin{align}
				x = -z_{xx}[1+C_{1}^{2}+z_{x}^{2}]^{-3/2}.
			\end{align}
			
			This equation can be explicitly solved and reduced to a first order differential equation by making the substitutions $ a^{2} = 1+C_{1}^{2} $ and $ z'(x) = a\tan(\theta) $, we then proceed as follows:\begin{align}
				\frac{1}{2}x^{2}+C_{2} &= -\int \frac{a\sec^{2}(\theta)}{a^{3}(1+\tan^{2}(\theta))^{3/2}}d\theta\\
				&= -\int \frac{1}{a^{2}}\cos(\theta)d\theta\\
				&= -\frac{1}{a^{2}}\sin(\theta)\\
				\frac{a^{2}}{2}(x^{2}+2C_{2})&= -\frac{z'(x)}{\sqrt{a^{2}+z'(x)^{2}}}\\
				\frac{1}{2}(1+C_{1}^{2})(x^{2}+2C_{2})&= -\frac{z'(x)}{\sqrt{1+C_{1}^{2}+z'(x)^{2}}}
			\end{align}
			and this is the ODE whose 2 parameter families of solutions provide rotating self-similar solutions around either the z or y-axis.
						
			\textbf{Rotation around the x-axis:}	For rotation around the x-axis we obtain the following system of equations:	
			\begin{align}
				0 &=(y'(x)z''(x)-z'(x)y''(x))[1+y_{x}^{2}+z_{x}^{2}]^{-3/2}\\
				z&=z''(x)[1+y_{x}^{2}+z_{x}^{2}]^{-3/2}\\
				y &= y''(x)[1+y_{x}^{2}+z_{x}^{2}]^{-3/2}.
			\end{align}
			
			The first equation it gives us that $ y'(x) = \lambda z'(x) $. This transforms the other equations into the second order non-linear differential equation: \[ z = z''[1+(\lambda^{2}+1)z_{x}^{2}]^{-3/2}. \]
			
			We now solve this equation and turn it into a first order nonlinear equation:\begin{align}
				z'z &= \frac{z'z''}{(1+(1+\lambda^{2})(z')^{2})^{3/2}}\\
				\frac{1}{2}z^{2}+C_{1} &= \int \frac{z'z''}{(1+(1+\lambda^{2})(z')^{2})^{3/2}}dx.\\			.
			\end{align}
			As in the previous problem we substitute $ u = (1+\lambda^{2})(z')^{2}(x) $, $ (1+\lambda^{2})^{-1}du = 2z'z'' dx$, continuing with the integration:\begin{align}
				(1+\lambda^{2})(z^{2}+2C_{1}) &= \int {(1+u)^{-3/2}}du.\\
				(1+\lambda^{2})(z^{2}+2C_{1}) &= -2(1+(1+\lambda^{2})(z')^{2})^{-1/2}\\
				(z')^{2}&=\frac{4}{(1+\lambda^{2})^{3}(z^{2}+2C_{1})^{2}}-\frac{1}{(1+\lambda^{2})}\\
				z'(x) &= \pm\sqrt{ \frac{4-(1+\lambda^2)^2(z^2 + 2C_1)^2}{(1+\lambda^2)^3(z^2 + 2C_1)^2}}
			\end{align}
				
			Finally, using this we can derive the expression for $ y'(x) $ using that $ y'(x) = \lambda z'(x) $, we have that $ y(x) = \lambda z(x)+C_{2} $. Plugging this into our previous expression gives:\begin{equation}
				y'(x) = \lambda \sqrt{ \frac{4-(1+\lambda^2)^2((\frac{y}{\lambda}+C_{2})^2 + 2C_1)^2}{(1+\lambda^2)^3((\frac{y}{\lambda}+C_{2})^2 + 2C_1)^2}}
			\end{equation} 
				
			Putting all results together gives a final expression for how the curves will behave if we assume the vortex filament ot be rotating around a particular axis.
		\end{proof}
		
		\begin{corollary}
			We find that if a vortex filament rotates inside of a singular plane it must obey the following second order non-linear differential equation \begin{align}
						f(x) +\frac{f''(x)}{(1+f'(x))^{3/2}} = 0
		\end{align}
		\end{corollary}
		
		\begin{proof}
			The general form for a curve rotating in a plane is of the form \begin{equation}
				\Gamma(x,t) = (x, \cos(t)f(x),-\sin(t)f(x))
			\end{equation}
			We now take the appropriate time and space derivatives \begin{align}
				\partial_{t}\Gamma &= (0, -\sin(t)f(x),-\cos(t)f(x))\\
				\partial_{x}\Gamma &= (1, \cos(t)f'(x),-\sin(t)f'(x))\\
				\partial_{xx}\Gamma&=(0, \cos(t)f''(x),-\sin(t)f''(x))
			\end{align}
			
			We can now plug these equations into our previous system of ODEs (insert equation numbers here) for rotating solutions in 3 dimensions by setting the left hand side as the components of $ \partial_{t}\Gamma(x,t) $ as well as taking $ y(x,t) = \cos(t)f(x) $ and $ z(x,t) = -\sin(t)f(x) $ and changing the right hand derivatives to partial derivatives with respect to x gives us:\begin{align}
				0 &= 0\\
				\partial_{t}y(x,t) &= \sin(t)f''(x)[1+f'(x)^{2}]^{-3/2}\\
				\partial_{t}z(x,t) &= \cos(t)f''(x)[1+f'(x)^{2}]^{-3/2}
			\end{align}
			
			Substituting the values for $ y_{t}(x,t) $ and $ z_{t}(x,t) $ gives the nonlinear differential equation: \begin{equation}
				f(x) +\frac{f''(x)}{(1+f'(x))^{3/2}} = 0
			\end{equation}
			We can further simplify this equation to turn it into a first degree ODE:\begin{align}
				ff' &= \frac{-f'f''}{(1+f'(x)^{2})^{3/2}}\\
				f(x)^{2}+2C_{1} &= -2(1+f'(x)^{2})^{-1/2}\\
				f'(x) &= \sqrt{\frac{4}{(f(x)^{2}+2C_{1})^{2}}-1}
			\end{align}
			
			This then gives the 2 parameter family of solutions of rotating solutions in the plane.
		\end{proof}
		
		\begin{remark}
			It is of interest to note that the previous result could have been achieved from theorem 3.4.1 by saying that it was a solution rotating around the x-axis with $ \lambda = 0 $. This then gives the same ODE governing the family of solutions for planar rotating vortex filaments. We also provide the following argument that demonstrates the solvability of equations (186), (187) and (199):
		\end{remark}
		
		Using the differentiability of $ z(x) $ we can create an inverse function such that now $ x $ is a function of z. Then from basic analysis we would have that the derivative with respect to z of the inverse function would then be \begin{equation}
			\frac{dx}{dz} = \sqrt{\frac{(1+\lambda^{2})^{3}(z^{2}+2C_{1})^{2}}{4-(1+\lambda^{2})^{2}(z^{2}+2C_{1})^{2}}}
		\end{equation}
		so that then x exists as a smooth increasing function on the range \[-\frac{2}{(1+\lambda^{2})}-2C_{1}<z^{2}<\frac{2}{(1+\lambda^{2})}-2C_{1} \] so then we have $ z(x) $ to be a smooth, increasing and bounded function of x.
	\section*{Acknowledgments}
		I would like to thank my graduate student advisor Jiewon Park for the direction and support in the development of this paper, as well as directing the goal of this paper during its development. I would also like to extend my gratitude towards Professor Tobias Colding for providing the initial papers that led to the development of this paper. I would finally extend my gratitude towards the MIT mathematics department and UROP+ coordinators, as well as the Undergraduate Research Opportunity office for providing funding over the summer towards the creation of this paper.
	
\end{document}